\documentclass[11pt,reqno,a4paper]{amsart}

\usepackage[T1]{fontenc}
\usepackage{amsmath,amsthm,amssymb}
\usepackage{microtype,fbb}
\usepackage{xcolor}
\usepackage[top=2.4cm,bottom=2.4cm,left=2.4cm,right=2.4cm,headsep=0.2in]{geometry}
\usepackage{fancyhdr}
\usepackage{hyperref}
\hypersetup{
  colorlinks=true,
  linkcolor=blue!50!green,
  citecolor=blue!50!green,
  urlcolor=blue!50!green,
  pdfauthor={Mayukh Mukherjee},
  pdftitle={A sharp covering theorem and Solyanik estimates for Euclidean balls}
}
\newtheorem{theorem}{Theorem}[section]
\newtheorem{maintheorem}{Theorem}

\newtheorem{lemma}[theorem]{Lemma}
\newtheorem{proposition}[theorem]{Proposition}
\newtheorem{corollary}[theorem]{Corollary}
\theoremstyle{remark}

\numberwithin{equation}{section}
\allowdisplaybreaks[1]
\title[A sharp covering theorem and Solyanik estimates]{A sharp covering theorem and Solyanik estimates for Euclidean balls}
\date{}
\keywords{Covering lemmas, Hardy--Littlewood maximal operator, Solyanik estimates}
\author{Mayukh Mukherjee}

\address{Department of Mathematics, Indian Institute of Technology Bombay,
Powai, Mumbai 400076, India}
\email{mukherjee@math.iitb.ac.in, mathmukherjee@gmail.com}

\begin{document}
\begin{abstract}
For every finite family of Euclidean balls in $\mathbb{R}^n$, $n\ge2$, and every $0<\delta<1/2$, we select a subfamily whose $(1+\delta)$-dilations cover the original union and whose undilated balls have multiplicity at most $C_n\delta^{-(n-1)/2}$. This proves the covering estimate conjectured by Han and Lu \cite{HL}. As an application, we determine the optimal Solyanik asymptotic
\[
\mathcal C_n(\alpha)-1\asymp_n(1-\alpha)^{2/(n+1)}
\qquad(\alpha\uparrow1)
\]
for the uncentered Hardy--Littlewood maximal operator over Euclidean balls, which is Conjecture~1(b) of Hagelstein and Parissis \cite{HP14}. For the modified uncentered maximal operators, we determine the optimal weak $(1,1)$ growth rate $(k-1)^{-(n-1)/2}$ as $k\downarrow1$, uniformly over Radon measures. We also obtain the corresponding $L^p$ and Fefferman--Stein bounds.
\end{abstract}

\maketitle

\section{Introduction}

Throughout the paper, $n\ge2$, all balls are open and have positive radii, and $aB(c,r)=B(c,ar)$ for $a>0$. Set $\sigma=(n-1)/2$ and $\rho=2/(n+1)=1/(\sigma+1)$. Constants denoted by $c_n$, $C_n$, or $A_n$ depend only on the dimension.

Chanillo and Muckenhoupt \cite[Lemma 3]{CM} proved that, for every finite family of balls in $\mathbb{R}^n$, one can select a subfamily whose $(1+\delta)$-dilations cover the original union and whose undilated balls have multiplicity at most $4^n\delta^{-n}$. In dimension two, Lu \cite[Lemma~1]{Lu91} improved the multiplicity bound to $C\delta^{-7/4}\log(1/\delta)$; he subsequently obtained the bound $C_n\delta^{-n+1/4}\log(1/\delta)$ in $\mathbb{R}^n$, $n\ge2$, in \cite[Lemma~C]{Lu93}. Sawano \cite[Theorem~1.5]{Saw} proved that, for each $k>1$, a family of Euclidean balls with uniformly bounded radii has a subfamily whose $k$-dilates cover the original union and which is the union of $N(n,k)$ families of pairwise disjoint balls. Han and Lu \cite[Theorem~1.5]{HL} later obtained $C_n\delta^{-n/2}\log(1/\delta)$ and conjectured that the optimal bound is $C_n\delta^{-(n-1)/2}$ \cite[Conjecture 1.6]{HL}. The following theorem proves that.

\begin{maintheorem}\label{thm:covering}
Let $0<\delta<1/2$. Every finite family $\mathcal{B}$ of balls in $\mathbb{R}^n$ has a subfamily $\mathcal{S}\subset\mathcal{B}$ such that
\begin{equation}\label{eq:covering}
\bigcup_{B\in\mathcal{B}}B\subset\bigcup_{B\in\mathcal{S}}(1+\delta)B
\end{equation}
and
\begin{equation}\label{eq:multiplicity}
\sum_{B\in\mathcal{S}}\mathbf{1}_B(x)\le A_n\delta^{-\sigma}\qquad(x\in\mathbb{R}^n).
\end{equation}
The exponent $\sigma=(n-1)/2$ is optimal, even for families of equal-radius balls.
\end{maintheorem}


We select the balls in nonincreasing order of radius and choose points $w_i\in B_i$ satisfying $w_i\notin B_j$ and $w_j\notin(1+\delta)B_i$ whenever $i<j$. Han and Lu \cite[Lemma~4.3 and the proof of Theorem~1.5]{HL} pack the Euclidean midpoints of centers and auxiliary points within each dyadic radius class, obtaining $\delta^{-n/2}$ from a packing in $\mathbb{R}^n$, and then sum over $O(\log(1/\delta))$ classes. At a common point of the selected balls, we extend each chosen point radially to the boundary and take the spherical midpoint of its radial direction and the outward normal reflected in the radial line. Lemma~\ref{lem:twoballs} shows that the exclusions in \eqref{eq:witnesses} separate these directions by a constant times $\sqrt\delta$ for arbitrary radius ratios, so a single packing on $S^{n-1}$ removes both the extra factor $\delta^{-1/2}$ and the logarithm.

For a locally integrable function on $\mathbb{R}^n$, let
\begin{equation}\label{eq:classical-maximal}
Mf(x)=\sup_{B\ni x}\frac1{|B|}\int_B|f(y)|\,dy,
\end{equation}
where the supremum is over Euclidean balls. For $0<\alpha<1$, define the sharp Tauberian constant
\begin{equation}\label{eq:classical-tauberian}
\mathcal{C}_n(\alpha)=\sup_{0<|E|<\infty}\frac{|\{M\mathbf{1}_E>\alpha\}|}{|E|}.
\end{equation}
Following Hagelstein and Parissis \cite[Section~1]{HP14}, we refer to estimates for $\mathcal{C}_n(\alpha)-1$ as $\alpha\uparrow1$ as Solyanik estimates, in reference to Solyanik's earlier work on the corresponding bordering (halo) functions \cite{Sol}. Hagelstein and Parissis \cite[Theorem~3]{HP14} proved $\mathcal{C}_n(\alpha)-1\lesssim_n(\alpha^{-1}-1)^{1/(n+1)}$, equivalently $O_n((1-\alpha)^{1/(n+1)})$ as $\alpha\uparrow1$. Their slab example \cite[Example~1]{HP14} gives the lower bound $\mathcal{C}_n(\alpha)-1\gtrsim_n(\alpha^{-1}-1)^{2/(n+1)}$, and they conjectured the matching two-sided asymptotic \cite[Conjecture~1(b)]{HP14}, see also \cite[Problem~8]{H24}.

\begin{maintheorem}\label{thm:solyanik}
There are constants $c_n,C_n>0$ and $\alpha_n\in(0,1)$ such that
\begin{equation}\label{eq:solyanik}
c_n(1-\alpha)^\rho\le\mathcal{C}_n(\alpha)-1\le C_n(1-\alpha)^\rho
\qquad(\alpha_n<\alpha<1).
\end{equation}
\end{maintheorem}

We also use the inequality
\begin{equation}\label{eq:union-dilation-intro}
\left|\bigcup_j aB_j\right|\le a^n\left|\bigcup_j B_j\right|\qquad(a\ge1),
\end{equation}
which also follows from Csik\'os \cite[Theorem~4.2]{Cs} by applying his contraction theorem to the linear motion of the centers from $c_j$ to $c_j/a$, after scaling the dilated union by $a^{-1}$. The same inequality appears in Dall'Ara \cite{DA}. We include a direct proof. Combining Theorem~\ref{thm:covering} with \eqref{eq:union-dilation-intro} gives
\[
\mathcal{C}_n(\alpha)-1\le C_n\bigl(\delta+(1-\alpha)\delta^{-\sigma}\bigr)
\]
for $\alpha$ close to $1$. The choice $\delta=(1-\alpha)^\rho$ yields the upper bound in Theorem~\ref{thm:solyanik}; the lower bound uses the same cap geometry as \cite[Example~1]{HP14}. The halo-set embedding of Hagelstein and Parissis \cite[Theorem~4]{HP15}, combined with Theorem~\ref{thm:solyanik}, implies that the associated halo function is locally H\"older continuous with exponent $\rho$; the endpoint exponent is sharp.

For a Radon measure $\mu$ on $\mathbb{R}^n$ and $k>1$, define
\[
M_{k,\mu}f(x)=\sup_{B\ni x}\frac{1}{\mu(kB)}\int_B|f|\,d\mu,
\]
with $0/0=0$. The undilated uncentered operator need not be of weak type $(1,1)$ when $n\ge2$, even for Gaussian measure on $\mathbb{R}^2$ \cite{Sj}; the dilated denominator restores the bound. Such operators appear in Nazarov, Treil, and Volberg \cite{NTV} for centered balls and in Tolsa \cite{To} for uncentered cubes. Sawano \cite[Theorems~1.6 and~3.1]{Saw} proved weak $(1,1)$ bounds and analogues of the weighted maximal inequality of Fefferman and Stein \cite{FS} for uncentered balls and every $k>1$. Theorem~\ref{thm:covering} gives the sharp dependence on $k$.

\begin{maintheorem}\label{thm:radon}
For $1<k\le3/2$, every Radon measure $\mu$ on $\mathbb{R}^n$, $f\in L^1(\mu)$, and $\lambda>0$,
\begin{equation}\label{eq:radon-weak}
\mu\{M_{k,\mu}f>\lambda\}
\le\frac{C_n(k-1)^{-\sigma}}{\lambda}\int |f|\,d\mu.
\end{equation}
The factor $(k-1)^{-\sigma}$ is sharp up to dimensional constants. For $1<p<\infty$,
\begin{equation}\label{eq:radon-strong}
\sup_\mu\|M_{k,\mu}\|_{L^p(\mu)\to L^p(\mu)}
\asymp_{n,p}(k-1)^{-\sigma/p},
\end{equation}
where the supremum runs over Radon measures on $\mathbb{R}^n$. Moreover, for $f,g\in L^1_{\mathrm{loc}}(\mu)$ with $g\ge0$,
\begin{equation}\label{eq:radon-fs}
\int (M_{k,\mu}f)^p g\,d\mu
\le C_{n,p}(k-1)^{-\sigma}\int |f|^p M_{\sqrt{k},\mu}g\,d\mu,
\end{equation}
and the power $\sigma$ is optimal. The lower bounds hold for a single probability measure $d\mu=w\,dx$ with $w\in C^\infty(\mathbb{R}^n)$ and $0<w\le C_n$, independent of $k$ and $p$.
\end{maintheorem}

Sections~\ref{sec:covering}-\ref{sec:radon} prove Theorems~\ref{thm:covering}-\ref{thm:radon}, respectively.

\section{The sharp covering theorem}\label{sec:covering}

\subsection{Selection of balls}

\begin{lemma}\label{lem:selection}
For every finite family $\mathcal{B}$ and every $\delta>0$, one can choose balls $B_i=B(c_i,r_i)\in\mathcal{B}$, $1\le i\le N$, and points $w_i\in B_i$ such that \eqref{eq:covering} holds with $\mathcal{S}=\{B_1,\ldots,B_N\}$, $r_1\ge r_2\ge\cdots\ge r_N,$ and, whenever $i<j$,
\begin{equation}\label{eq:witnesses}
w_i\notin B_j,\qquad w_j\notin(1+\delta)B_i.
\end{equation}
\end{lemma}

\begin{proof}
We construct the balls successively. Let $\mathcal{S}$ denote the balls already selected and $\mathcal{R}$ those still under consideration. Initially $\mathcal{S}=\varnothing$ and $\mathcal{R}=\mathcal{B}$, and throughout the construction we keep
\begin{equation}\label{eq:invariant}
\bigcup_{B\in\mathcal{B}}B
\subset
\bigcup_{B\in\mathcal{S}}(1+\delta)B
\cup
\bigcup_{B\in\mathcal{R}}B.
\end{equation}

First remove from $\mathcal{R}$ any ball $D$ which is already covered by the other sets, that is, any $D$ satisfying
\begin{equation}\label{eq:deletion}
D\subset
\bigcup_{B\in\mathcal{S}}(1+\delta)B
\cup
\bigcup_{B\in\mathcal{R}\setminus\{D\}}B.
\end{equation}
Continue until no such ball remains. This does not affect \eqref{eq:invariant}.

If $\mathcal{R}$ is nonempty, choose a ball $B_i$ of largest radius in $\mathcal{R}$. Since $B_i$ cannot be removed, there is a point
\[
w_i\in
B_i\setminus
\left(
\bigcup_{B\in\mathcal{S}}(1+\delta)B
\cup
\bigcup_{D\in\mathcal{R}\setminus\{B_i\}}D
\right).
\]
Add $B_i$ to $\mathcal{S}$ and remove it from $\mathcal{R}$. Then return to the deletion step above and continue.

Since $\mathcal{B}$ is finite, the construction eventually stops. At that point $\mathcal{R}=\varnothing$, so \eqref{eq:invariant} gives \eqref{eq:covering}. Because at each stage we choose a largest remaining ball, the selected radii are nonincreasing.

Finally, suppose $i<j$. When $w_i$ was chosen, the later ball $B_j$ was still in $\mathcal{R}$, so $w_i\notin B_j$. When $w_j$ was chosen, $B_i$ had already been selected, so $w_j\notin(1+\delta)B_i$. 
\end{proof}
\subsection{An angular estimate for two balls}

\begin{lemma}\label{lem:midpoints}
Let $u,m,v,\ell\in S^{n-1}$ satisfy $\alpha=u\cdot m\ge0$ and $\beta=v\cdot\ell\ge0$. Set
\[
z=\frac{u+m}{|u+m|},\qquad z'=\frac{v+\ell}{|v+\ell|},\qquad h=|z-z'|^2,
\]
and
\[
P=m\cdot v,\qquad Q=\ell\cdot u,\qquad T=\alpha+\beta-P-Q,\qquad D=\alpha\beta-PQ.
\]
Then
\begin{equation}\label{eq:midpointbounds}
T\le h,\qquad D\le2h.
\end{equation}
\end{lemma}

\begin{proof}
For some $0\le\theta,\phi\le\pi/4$ and unit vectors $e\perp z$, $f\perp z'$, write
\begin{align*}
u&=\cos\theta\,z+\sin\theta\,e,&
m&=\cos\theta\,z-\sin\theta\,e,\\
v&=\cos\phi\,z'+\sin\phi\,f,&
\ell&=\cos\phi\,z'-\sin\phi\,f.
\end{align*}
Thus $\alpha=\cos(2\theta)$ and $\beta=\cos(2\phi)$. Put
\begin{align*}
H&=\cos(\theta+\phi),\\
A&=\cos\theta\cos\phi\,z\cdot z'-\sin\theta\sin\phi\,e\cdot f,\\
X&=\cos\theta\sin\phi\,z\cdot f-\sin\theta\cos\phi\,e\cdot z'.
\end{align*}
Then $P=A+X$ and $Q=A-X$. Since $z\cdot z'=1-h/2$ and $e\cdot f\le1$,
\[
A\ge H-\tfrac12\cos\theta\cos\phi\,h\ge H-h/2.
\]
Also $|z\cdot f|,|e\cdot z'|\le\sqrt h$, whence
\[
|X|\le\sin(\theta+\phi)\sqrt h\le\sqrt h.
\]
The identities
\[
\alpha+\beta=2H\cos(\theta-\phi)\le2H,\qquad
\alpha\beta=H^2-\sin^2(\theta-\phi)\le H^2
\]
give $T=\alpha+\beta-2A\le h$. Finally, $0\le H\le1$ implies
\[
A^2\ge(H-h/2)_+^2\ge H^2-h.
\]
Consequently $D=\alpha\beta-A^2+X^2\le2h$.
\end{proof}

Let $B=B(c,r)$ contain the origin, and let $p\in\partial B$. Define
\begin{equation}\label{eq:directions}
u=\frac p{|p|},\qquad \nu=\frac{p-c}{r},\qquad
\alpha=u\cdot\nu,\qquad m=2\alpha u-\nu,\qquad
z(B,p)=\frac{u+m}{|u+m|}.
\end{equation}
These quantities are well defined. Indeed,
\begin{equation}\label{eq:acute}
0<r^2-|c|^2=2r|p|\alpha-|p|^2
\end{equation}
gives $\alpha>0$ and $|p|/r<2\alpha$. Moreover $|m|=1$ and $u\cdot m=\alpha$, so $u+m\ne0$.

The ball through $0$ and $p$ which is internally tangent to $B$ at $p$ has center $r'm$, where $r'=|p|/(2\alpha)$. Thus $m$ is the direction of its center, and $z(B,p)$ is the midpoint of the shorter spherical arc joining $u$ and $m$. The next lemma shows that, if $r_1\ge r_2$, $p_1\notin B_2$, and $p_2\notin(1+\delta)B_1$, then
\[
|z(B_1,p_1)-z(B_2,p_2)|^2\ge\frac{\delta}{10}.
\]
This reduces the required estimate for collections of balls to a packing argument on $S^{n - 1}$.

\begin{lemma}\label{lem:twoballs}
Let $B_i=B(c_i,r_i)$ contain the origin, let $r_1\ge r_2$, and let $p_i\in\partial B_i$, $i=1,2$. Suppose $p_1\notin B_2$. Then
\begin{equation}\label{eq:twoball}
\frac{|p_2-c_1|^2-r_1^2}{r_1^2}
\le20\,|z(B_1,p_1)-z(B_2,p_2)|^2.
\end{equation}
In particular, if $p_2\notin(1+\delta)B_1$, then
\begin{equation}\label{eq:separation}
|z(B_1,p_1)-z(B_2,p_2)|^2\ge\frac{2\delta+\delta^2}{20}\ge\frac\delta{10}.
\end{equation}
\end{lemma}

\begin{proof}
Use \eqref{eq:directions} for the first ball, with notation $u,\nu_1,\alpha,m$, and for the second ball, with notation $v,\nu_2,\beta,\ell$. Thus $\alpha=u\cdot m>0$ and $\beta=v\cdot\ell>0$. Write
\[
P=m\cdot v,\quad Q=\ell\cdot u,\quad \gamma=u\cdot v,\quad
h=|z(B_1,p_1)-z(B_2,p_2)|^2.
\]
Lemma~\ref{lem:midpoints} gives
\begin{equation}\label{eq:TD}
T:=\alpha+\beta-P-Q\le h,\qquad D:=\alpha\beta-PQ\le2h.
\end{equation}

Replace $B_2$ by the internally tangent ball
\[
B'_2=B(c'_2,r'_2),\qquad r'_2=\frac{|p_2|}{2\beta},\qquad
c'_2=p_2-r'_2\nu_2=r'_2\ell.
\]
By \eqref{eq:acute}, $r'_2<r_2$. The centers satisfy $|c'_2-c_2|=r_2-r'_2$, so $B'_2\subset B_2$. Its boundary contains both $0$ and $p_2$. This replacement leaves $v,\nu_2,\beta,\ell$ unchanged, and $p_1\notin B'_2$.

Normalize by $r_1$ and set $s=|p_1|/r_1$ and $t=|p_2|/r_1$. Since $r'_2\le r_2\le r_1$, and since $p_1\notin B'_2$, we have
\begin{equation}\label{eq:polygon}
0\le s\le2\alpha,\qquad 0\le t\le2\beta,\qquad \beta s\ge Qt.
\end{equation}
For the last inequality, use $c'_2=r'_2\ell$ to obtain $|p_1|^2-2|p_1|r'_2Q\ge0$, divide by $|p_1|>0$, and substitute $r'_2/r_1=t/(2\beta)$.

Because $c_1/r_1=(s-2\alpha)u+m$, the left side of \eqref{eq:twoball} equals
\begin{equation}\label{eq:F}
F(s,t)=t^2-2tP+(s-2\alpha)(s-2t\gamma).
\end{equation}
This is a convex quadratic in $(s,t)$: it is the squared norm of the affine vector $tv-(s-2\alpha)u-m$, minus $1$. Its maximum over the polygon \eqref{eq:polygon} is therefore attained at a vertex. The vertices $(0,0)$ and $(2\alpha,0)$ give $F=0$. The other possible vertices are treated below.

\smallskip
\noindent\textit{The vertex $(2\alpha,2\beta)$, when $Q\le\alpha$.}
Here
\[
\frac F4=\beta(\beta-P)\le\beta T\le h,
\]
since $T=(\alpha-Q)+(\beta-P)$ and $0<\beta\le1$.

\smallskip
\noindent\textit{The vertex $(2Q,2\beta)$, when $0<Q\le\alpha$.}
Put $L=\alpha-Q\ge0$. Direct substitution gives
\begin{align}
\frac F4
&=\beta T+L(\beta-Q)-2\beta L(1-\gamma)\label{eq:vertex2a}\\
&=D+(\beta-Q)T-2\beta L(1-\gamma),\label{eq:vertex2b}
\end{align}
where the second identity uses $QT+L(\beta-Q)=D$. If $\beta<Q$, \eqref{eq:vertex2a} gives $F/4\le h$. If $\beta\ge Q$, then $0\le\beta-Q\le1$, and \eqref{eq:vertex2b} gives $F/4\le3h$.

\smallskip
\noindent\textit{The vertex $(2\alpha,2\alpha\beta/Q)$, when $Q>\alpha$.}
Here
\[
\frac F4=\frac{\alpha\beta}{Q^2}D.
\]
If $\beta\le Q$, the coefficient is at most $1$, and $F/4\le2h$. If $\beta>Q$, the identity
\[
D=QT+(Q-\alpha)(Q-\beta)\le Qh
\]
gives $F/4\le(\alpha\beta/Q)h\le h$.

\smallskip
\noindent\textit{The vertex $(0,2\beta)$, when $Q\le0$.}
Here
\[
\frac F4=\beta(\beta-P)+2\alpha\beta\gamma.
\]
The first term is at most $h$. If $P\ge0$, then $PQ\le0$ and $\alpha\beta\le D\le2h$. If $P<0$, then $T\ge\alpha+\beta\ge\alpha\beta$, so $\alpha\beta\le h$. In either case $F/4\le h+4h=5h$.

\smallskip
These are all the vertices: \eqref{eq:polygon} is a rectangle for $Q\le0$, a quadrilateral for $0<Q<\alpha$, and a triangle for $Q\ge\alpha$, with coincident vertices at equality. Hence $F\le20h$, proving \eqref{eq:twoball}. If $p_2\notin(1+\delta)B_1$, its left side is at least $(1+\delta)^2-1$, which proves \eqref{eq:separation}.
\end{proof}

\subsection{Proof of Theorem~\ref{thm:covering}}

\begin{proof}
Apply Lemma~\ref{lem:selection}. It remains to bound the multiplicity of the selected original balls. Fix $x\in\mathbb{R}^n$. Let $B_{i_1},\ldots,B_{i_M}$ be the selected balls containing $x$, in their inherited order. There is nothing to prove if $M\le1$. Translate the balls and the corresponding points from Lemma~\ref{lem:selection} by $-x$, and relabel them as $B_1,\ldots,B_M$ and $w_1,\ldots,w_M$, respectively, preserving the order.

The points $w_i$ are nonzero. Indeed, $w_1\notin B_2$ while $0\in B_2$, and, for $i\ge2$, $w_i\notin(1+\delta)B_1$ while $0\in(1+\delta)B_1$. Since $0,w_i\in B_i$, the ray from $0$ through $w_i$ meets $\partial B_i$ at a unique point $p_i$ beyond $w_i$.

If $i<j$ and $p_i\in B_j$, then the convexity of $B_j$, together with $0\in B_j$, would imply that the segment $[0,p_i]$, which contains $w_i$, lies in $B_j$. This contradicts $w_i\notin B_j$. Hence $p_i\notin B_j$. The same argument, using $0\in(1+\delta)B_i$, gives $p_j\notin(1+\delta)B_i$.

For $i<j$, we have $r_i\ge r_j$, so Lemma~\ref{lem:twoballs} yields
\[
|z(B_i,p_i)-z(B_j,p_j)|\ge\sqrt{\delta/10}.
\]
Thus the points $z(B_i,p_i)$, $1\le i\le M$, are pairwise separated in Euclidean distance by $\sqrt{\delta/10}$. The spherical caps of chordal radius $\tfrac13\sqrt{\delta/10}$ centered at these points are pairwise disjoint. Each has $(n-1)$-dimensional surface measure at least $c_n\delta^{(n-1)/2}$. Comparing with the surface measure of $S^{n-1}$ gives $M\le C_n\delta^{-(n-1)/2}$. The optimality assertion follows from Proposition~\ref{prop:sharpness} below.
\end{proof}

The lower-bound construction below is a quantitative variant of \cite[Example~5.1]{HL}.

\begin{proposition}\label{prop:sharpness}
For $0<\delta<1/2$, there are $N$ unit balls $B_i$ and distinct nonzero points $q_i$, $1\le i\le N$, such that
\begin{equation}\label{eq:sharp-config}
N\ge c_n\delta^{-\sigma},\qquad 0,q_i\in B_i,\qquad
q_j\notin(1+\delta)B_i\quad(i\ne j).
\end{equation}
Every subfamily satisfying \eqref{eq:covering} for this family therefore has multiplicity $N$ at the origin.
\end{proposition}

\begin{proof}
For $0<\delta<1/64$, choose a maximal $4\sqrt\delta$-separated set $\{\omega_1,\ldots,\omega_N\}\subset S^{n-1}$. The caps of chordal radius $4\sqrt\delta$ cover $S^{n-1}$, so $N\ge c_n\delta^{-(n-1)/2}$. Set
\[
B_i=B\left(\tfrac12\omega_i,1\right),\qquad
q_i=\left(\tfrac32-\delta\right)\omega_i.
\]
All the balls contain the origin, and $q_i\in B_i$. For $j\ne i$,
\begin{align*}
\left|q_i-\tfrac12\omega_j\right|^2
&=(1-\delta)^2+\left(\tfrac34-\tfrac\delta2\right)|\omega_i-\omega_j|^2\\
&\ge1+10\delta-7\delta^2>(1+\delta)^2.
\end{align*}
 Each $B_i$ must be retained because $q_i$ is not covered by the enlargement of another ball. For $1/64\le\delta<1/2$, one ball containing $0$ and a nonzero point gives the same conclusion after decreasing $c_n$.
\end{proof}

\section{Solyanik estimates for the ball maximal operator}\label{sec:solyanik}

\subsection{Dilation of unions}

The following inequality also follows from \cite[Theorem~4.2]{Cs}, applied to the linear contraction of the centers described in the introduction; see also \cite{DA}. We include an elementary proof.

\begin{lemma}\label{lem:union-dilation}
Let $B_j=B(c_j,r_j)$, $1\le j\le N$, and let $a\ge1$. Then
\begin{equation}\label{eq:union-dilation}
\left|\bigcup_{j=1}^N aB_j\right|\le a^n\left|\bigcup_{j=1}^N B_j\right|.
\end{equation}
\end{lemma}

\begin{proof}
The case $a=1$ follows by partitioning the union among its balls. Suppose $a>1$, and first retain only the largest ball at each center. This changes neither union. Let $V=\bigcup_j aB_j$ and $T_j(x)=c_j+(x-c_j)/a$. For $x\in V$, choose an index minimizing
\[
|x-c_j|^2-a^2r_j^2,
\]
breaking ties by the smallest index, and let $V_j$ be the set of points assigned to $j$. These sets form a Borel partition of $V$. Since the minimum is negative on $V$, one has $V_j\subset aB_j$, and consequently $T_j(V_j)\subset B_j$.

For $x\in V_i$ and $y\in V_j$, the two minimizing inequalities give
\begin{equation}\label{eq:cell-monotonicity}
(x-y)\cdot(c_i-c_j)\ge0.
\end{equation}
If $T_i(x)=T_j(y)$, then $x-y=-(a-1)(c_i-c_j)$. For distinct centers this contradicts \eqref{eq:cell-monotonicity}. Thus the images are pairwise disjoint. Since each $T_j$ has Jacobian $a^{-n}$,
\[
a^{-n}|V|=\sum_j|T_j(V_j)|\le\left|\bigcup_jB_j\right|.
\]
\end{proof}

\subsection{Proof of Theorem~\ref{thm:solyanik}}

\begin{proof}
Fix a measurable set $E$ with $0<|E|<\infty$, and write $H_\alpha(E)=\{M\mathbf{1}_E>\alpha\}$. This is an open set, being the union of all balls satisfying $|B\cap E|>\alpha|B|$. A compact set $F\subset H_\alpha(E)$ is covered by finitely many such balls. Apply Theorem~\ref{thm:covering} with $0<\delta<1/2$ and write
\[
U=\bigcup_jB_j,\qquad V=\bigcup_j(1+\delta)B_j
\]
for the selected family. Then $F\subset V$ and
\begin{align*}
|U\setminus E|
&\le\sum_j|B_j\setminus E|\\
&\le\frac{1-\alpha}{\alpha}\sum_j|B_j\cap E|
\le A_n\delta^{-\sigma}\frac{1-\alpha}{\alpha}|E|.
\end{align*}
Lemma~\ref{lem:union-dilation} therefore gives
\[
|F|\le(1+\delta)^n\left(1+A_n\delta^{-\sigma}\frac{1-\alpha}{\alpha}\right)|E|.
\]
Taking the supremum over compact $F$ and then over $E$ gives
\begin{equation}\label{eq:tauberian-parameter}
\mathcal{C}_n(\alpha)\le(1+\delta)^n\left(1+A_n\delta^{-\sigma}\frac{1-\alpha}{\alpha}\right)
\qquad(0<\alpha<1,\ 0<\delta<1/2).
\end{equation}
For $\alpha\ge1/2$, this implies
\[
\mathcal{C}_n(\alpha)-1\le C_n\bigl(\delta+(1-\alpha)\delta^{-\sigma}\bigr).
\]
Choose $\delta=(1-\alpha)^\rho$ when $\alpha$ is close enough to $1$ that $\delta<1/2$. Since $\rho(\sigma+1)=1$, the upper bound in \eqref{eq:solyanik} follows.

For the lower bound, we use a localized and rescaled variant of the slab example in \cite[Example~1]{HP14}. Let
\[
E=(-3,3)^{n-1}\times(-2,0),\qquad
S_h=(-1,1)^{n-1}\times(0,h),
\]
where $0<h<1/4$. For $x'\in(-1,1)^{n-1}$, the ball $B((x',h-1),1)$ contains every point $(x',t)$ with $0<t<h$. Its part outside $E$ is its upper cap of height $h$. Writing $\kappa_k=|B_{\mathbb{R}^k}(0,1)|$, the cap volume satisfies
\begin{equation}\label{eq:cap-volume}
\kappa_{n-1}\int_0^h(2u-u^2)^{(n-1)/2}\,du
\le\frac{2^{(n+1)/2}\kappa_{n-1}}{n+1}h^{(n+1)/2}.
\end{equation}
Thus, for some $L_n>0$, the average of $\mathbf{1}_E$ on each of these balls is at least $1-L_nh^{(n+1)/2}$. Choose $h=c_n(1-\alpha)^\rho$ with $c_n$ small enough that this average is greater than $\alpha$. Then $S_h\subset H_\alpha(E)$. Also $E\subset H_\alpha(E)$, since $E$ is open. Therefore
\[
\mathcal{C}_n(\alpha)-1\ge\frac{|S_h|}{|E|}=\frac{h}{2\cdot3^{n-1}}
\ge c_n(1-\alpha)^\rho.
\]
\end{proof}

\subsection{The halo function}

Following \cite[Section~1]{HP15}, define $\phi_n(t)=t$ for $0\le t\le1$ and $\phi_n(t)=\mathcal{C}_n(1/t)$ for $t>1$.

\begin{corollary}\label{cor:halo}
The functions $\mathcal{C}_n$ and $\phi_n$ are locally H\"older continuous with exponent $\rho$ on $(0,1)$ and $(1,\infty)$, respectively. Moreover,
\begin{equation}\label{eq:halo-endpoint}
\phi_n(1+h)-1\asymp_n h^\rho\qquad(h\downarrow0).
\end{equation}
In particular, no exponent larger than $\rho$ gives a pointwise H\"older bound for $\phi_n$ at $1$.
\end{corollary}

\begin{proof}
Write $H_\alpha(E)=\{M\mathbf{1}_E>\alpha\}$. By \cite[Theorem~4]{HP15}, there are dimensional constants $a_n,b_n,d_n>0$ such that
\begin{equation}\label{eq:halo-embedding}
H_\alpha(E)\subset H_{\alpha(1+a_n\min\{\alpha,1-\alpha\}^{2n}\eta)}
\bigl(H_{1-b_n\eta}(E)\bigr)
\end{equation}
whenever $0<\alpha<1$ and $0<\eta\le d_n(1-\alpha)$.

Fix a compact interval $I\subset(0,1)$. Estimate \eqref{eq:tauberian-parameter}, with $\delta=1/4$, shows that $\mathcal{C}_n$ is bounded on $I$ and that every halo set occurring below has finite measure. Let $\alpha<\beta$ belong to $I$ and suppose that $\beta-\alpha$ is sufficiently small. Set
\[
\eta=\frac{\beta/\alpha-1}
{a_n\min\{\alpha,1-\alpha\}^{2n}}.
\]
Then $\eta\asymp_{n,I}\beta-\alpha$, and, if $\beta-\alpha$ is sufficiently small, \eqref{eq:halo-embedding} applies and $1-b_n\eta>\alpha_n$. Taking measures in \eqref{eq:halo-embedding} and then the supremum over $E$ gives
\[
\mathcal{C}_n(\alpha)\le \mathcal{C}_n(\beta)\mathcal{C}_n(1-b_n\eta).
\]
Since $\mathcal{C}_n$ is nonincreasing, Theorem~\ref{thm:solyanik} yields
\[
0\le \mathcal{C}_n(\alpha)-\mathcal{C}_n(\beta)
\le \mathcal{C}_n(\beta)\bigl(\mathcal{C}_n(1-b_n\eta)-1\bigr)
\le C_{n,I}(\beta-\alpha)^\rho.
\]
For larger values of $\beta-\alpha$, the same estimate follows from the boundedness of $\mathcal{C}_n$ on $I$. Hence $\mathcal{C}_n\in C^\rho_{\mathrm{loc}}((0,1))$. Since $t\mapsto1/t$ is Lipschitz on compact subintervals of $(1,\infty)$, it follows that $\phi_n\in C^\rho_{\mathrm{loc}}((1,\infty))$.

Finally, \eqref{eq:solyanik} gives
\[
\phi_n(1+h)-1
=\mathcal{C}_n\!\left(\frac1{1+h}\right)-1
\asymp_n\left(1-\frac1{1+h}\right)^\rho
\asymp_n h^\rho
\qquad(h\downarrow0).
\]
The lower bound excludes every larger pointwise H\"older exponent at $1$.
\end{proof}

\section{Modified maximal operators for Radon measures}\label{sec:radon}

\begin{proof}[Proof of Theorem~\ref{thm:radon}]
Put $t=\sqrt{k}$. If $\int_B|f|\,d\mu>\lambda\mu(kB)$, then $M_{t,\mu}g(y)\ge\mu(kB)^{-1}\int_{tB}g\,d\mu$ for $y\in B$, and hence
\[
\int_{tB}g\,d\mu\le\frac1\lambda\int_B|f|M_{t,\mu}g\,d\mu.
\]
The open set $\{M_{k,\mu}f>\lambda\}$ is the union of all such balls. Cover a compact subset by finitely many of them and apply Theorem~\ref{thm:covering} with $\delta=t-1$. Summing the last inequality and using the multiplicity bound, then taking the supremum over compact subsets, gives
\begin{equation}\label{eq:radon-distribution}
\int_{\{M_{k,\mu}f>\lambda\}}g\,d\mu
\le\frac{C_n(k-1)^{-\sigma}}{\lambda}\int |f|M_{t,\mu}g\,d\mu.
\end{equation}
Since $M_{k,\mu}(f\mathbf{1}_{\{|f|\le\lambda/2\}})\le\lambda/2$, apply \eqref{eq:radon-distribution} to $f\mathbf{1}_{\{|f|>\lambda/2\}}$ at level $\lambda/2$ and integrate against $p\lambda^{p-1}\,d\lambda$. This proves \eqref{eq:radon-fs}. Taking $g=1$, for which $M_{t,\mu}g\le1$, proves \eqref{eq:radon-weak} and the upper bound in \eqref{eq:radon-strong}.

The construction below combines the atomic example of \cite{GK} with the configurations of Proposition~\ref{prop:sharpness}, replacing the point masses by smooth bumps.

For sharpness, take the configurations $B_{j,i},q_{j,i}$ from the proof of Proposition~\ref{prop:sharpness} with $\delta_j=2^{-j}$, $j\ge7$, and $N_j\ge c_n\delta_j^{-\sigma}$. Put $q_{j,0}=0$ and $\eta_j=\delta_j/100$. The estimates there show that $B(0,\eta_j)$ and $B(q_{j,i},\eta_j)$ lie in $B_{j,i}$, while $B(q_{j,l},\eta_j)$ is disjoint from $(1+\delta_j)B_{j,i}$ for $l\ne0,i$.

Fix $0\le\psi\in C_c^\infty(B(0,1))$ with integral one, and set $b_j=2^{-j}\eta_j^n/(N_j+1)$. Choose $a_j\in\mathbb{R}^n$ tending to infinity so that the balls $B(a_j,2)$ are pairwise disjoint and $\int_{B(a_j,2)}e^{-|x|^2}\,dx\le b_j$. Define
\[
w_0(x)=e^{-|x|^2}+\sum_{j\ge7}b_j\eta_j^{-n}
\sum_{i=0}^{N_j}\psi\!\left(\frac{x-a_j-q_{j,i}}{\eta_j}\right).
\]
The sum is locally finite, has integral $\sum_j2^{-j}\eta_j^n$, and is bounded by $\|\psi\|_\infty\sum_j2^{-j}$. Thus $w_0$ is smooth, positive, bounded, and integrable. Normalize $d\mu_0=w_0\,dx$ to a probability measure $\mu$.

Given $0<k-1\le2^{-7}$, choose $j$ with $k-1\le\delta_j<2(k-1)$ and set $f_j=\mathbf{1}_{B(a_j,\eta_j)}$. Every $k(a_j+B_{j,i})$ lies in $B(a_j,2)$ and contains exactly the two bumps at $a_j$ and $a_j+q_{j,i}$. Its measure is at most $3b_j$, and its numerator is at least $b_j$. Consequently $M_{k,\mu_0}f_j\ge1/3$ on $\bigcup_{i=1}^{N_j}B(a_j+q_{j,i},\eta_j)$, so
\[
\mu_0\{M_{k,\mu_0}f_j>1/4\}\ge N_jb_j,
\qquad
\|f_j\|_{L^1(\mu_0)}=\|f_j\|_{L^p(\mu_0)}^p\le2b_j.
\]
These estimates prove sharpness in \eqref{eq:radon-weak} and the lower bound in \eqref{eq:radon-strong}; normalization does not change the weak $(1,1)$ or $L^p$ constants. They also prove sharpness in \eqref{eq:radon-fs} by taking $g=1$. The remaining range of $k$ follows from $M_{k,\mu}1=1$, since $\mu$ is finite.
\end{proof}

\end{document}